\documentclass[11pt]{article}
\usepackage[T1]{fontenc}
\usepackage[utf8]{inputenc}
\usepackage{lmodern}
\usepackage{amsmath,amssymb,amsthm,mathtools}
\usepackage{microtype}
\usepackage{geometry}
\usepackage{xcolor}
\usepackage{hyperref}
\usepackage{enumitem}
\usepackage{booktabs}
\usepackage{array}
\hypersetup{colorlinks=true,linkcolor=blue!45!black,citecolor=blue!45!black,urlcolor=blue!45!black,pdftitle={A direct proof of Infinity in NF^(-empty set)},pdfauthor={Zuhair A. Al-Johar}}
\newtheorem{theorem}{Theorem}[section]
\newtheorem{lemma}[theorem]{Lemma}

\newtheorem{corollary}[theorem]{Corollary}

\newtheorem{definition}[theorem]{Definition}
\newcommand{\NF}{\mathrm{NF}^{-\varnothing}}

\newcommand{\Fin}{\operatorname{Finite}}
\newcommand{\Even}{\operatorname{Even}}
\newcommand{\BiWO}{\operatorname{BiWO}}
\newcommand{\NoNT}{\operatorname{NoNT}}
\newcommand{\USC}{\operatorname{USC}}
\newcommand{\Powp}{\mathcal{P}^{+}}
\newcommand{\Orb}{\Phi}
\newcommand{\Succ}{\operatorname{Succ}}
\newcommand{\defeq}{\mathrel{:=}}

\title{A direct proof of Infinity in $\mathrm{NF}^{-\varnothing}$}
\author{Zuhair A. Al-Johar}
\date{17 October 2026}
\begin{document}
\maketitle

\begin{abstract}
We give a direct proof of Infinity in the no-empty version of New Foundations, denoted $\NF$, by establishing that every model of the specified theory contains an object that is not finite in a precise, internally defined sense. The version considered has extensionality, at least two distinct objects, no empty objects, and \emph{inhabited stratified comprehension}: a stratified extension is supplied whenever it has a member. Finiteness combines a two-sided internally represented well-order with the exclusion of internally represented nonterminating strict transitive relations on inhabited subsets. Assuming all objects finite, partial bijections yield comparison and finite extrema of positive Frege cardinals. Positive exponentiation, singleton type-raising, and the least orbit of the genuine two-cardinal lead to the equation $n=T(n)+1_F$. We prove separately, and in detail, that adjoining one point reverses internally defined parity: the key ingredients are a graph-rewiring proof of invariance under adjoining two points and one instance of stratified induction derived from the all-finite assumption. Since type-raising preserves parity, the cardinal equation is impossible. We identify the corresponding Lean~4 declarations and make explicit the scope of the formal verification. No model-existence or consistency theorem for ordinary NF is claimed.
\end{abstract}


\section{Introduction and exact scope}
Quine's New Foundations (NF) uses stratified comprehension in a one-sorted language \cite{Quine1937}. Specker's classical work on Choice and Infinity in NF \cite{Specker1953}, and Rosser's analysis of Infinity and ordered pairs \cite{Rosser1952}, motivate the type-raising and cardinal-comparison viewpoint used here. Our result concerns the \emph{no-empty version of NF}, denoted $\NF$, with the inhabited-comprehension axiom specified below, rather than an independently constructed model of ordinary NF or NFU. In particular, the adjective ``positive'' in the accompanying formal development refers to the \emph{inhabited-output restriction on comprehension}, not to a prohibition of logical negation in stratified formulas.

Here is the exact assertion proved in the formal development:
\begin{theorem}[Infinity for the no-empty version of NF]\label{thm:main}
Let $M$ be a one-sorted membership structure satisfying extensionality, inhabited stratified comprehension, the existence of at least two distinct objects, and the assertion that every object has a member. Then
\[
\exists S\;\neg\Fin(S),
\]
where $\Fin$ is Definition~\ref{def:finite} below.
\end{theorem}

The result is an implication from the stated model axioms. It does not assert that a model of these axioms exists; nor does it assert $\mathrm{Con}(\mathrm{NF})$. It also does not identify the displayed finiteness predicate with every other notion of finiteness unless the relevant equivalence is separately invoked. The supplied formalization contains an independent Tarski-finiteness appendix, which is not required for Theorem~\ref{thm:main}.

\section{Theory, coding, and finite sets}\label{sec:foundations}
Let $V$ denote the internal universal set, existing because $x=x$ is stratified and inhabited. We write $x\in A$ for membership and $A\subseteq B$ for $\forall x(x\in A\Rightarrow x\in B)$. The model assumptions are:
\begin{enumerate}[label=(\roman*),leftmargin=2em]
\item \emph{Extensionality:} sets with exactly the same members are equal.
\item \emph{Nontriviality:} there are distinct objects $a\neq b$.
\item \emph{No empty objects:} $\forall A\;\exists x\;(x\in A)$.
\item \emph{Inhabited stratified comprehension:} for a stratified formula $\varphi(x,\vec p)$, if $\exists x\,\varphi(x,\vec p)$ then there is a set $G$ satisfying $x\in G\leftrightarrow\varphi(x,\vec p)$.
\end{enumerate}

Quantification is over the model's objects. A formula is stratified if integer levels may be assigned to its variables such that equality preserves level and the right argument of membership has one more level than the left. The definition permits negation and quantification in $\varphi$; only \emph{inhabitedness of its extension} is demanded.

Singletons, unordered pairs, unions and nonempty differences are used only when constructed by an appropriate inhabited stratified instance. Relations and bijections mean \emph{actual internally represented sets of homogeneous Kuratowski pair codes}; no arbitrary external function is automatically an internal bijection. For sets $A,B$ let $A\simeq B$ mean that there is an internal bijection graph from $A$ to $B$.

\begin{definition}[Two-sided well-order]\label{def:biwo}
$\BiWO(R,A)$ means that the internally coded reflexive relation $R$ is linear, antisymmetric and transitive on $A$, and that every inhabited internally represented subset of $A$ has both an $R$-least and an $R$-greatest element. The order graph is a set of pair codes.
\end{definition}

\begin{definition}[Nontermination prohibition]\label{def:nont}
$\NoNT(A)$ asserts that no inhabited internally represented $D\subseteq A$ admits an internally represented relation $R$ on $D$ that is simultaneously irreflexive, transitive and nonterminating:
\[
\forall x\in D\;\exists y\in D\;R(x,y).
\]
\end{definition}

\begin{definition}[Finiteness]\label{def:finite}
\begin{equation}\label{eq:finite}
\Fin(A)\ \longleftrightarrow\
 \bigl(\exists R\,\BiWO(R,A)\bigr)\land\NoNT(A).
\end{equation}
\end{definition}

The proof proceeds by contradiction from
\begin{equation}\label{eq:allfinite}
\forall A\;\Fin(A).
\end{equation}
The order part permits standard finite minimum/maximum arguments; the $\NoNT$ part rules out an internally represented infinite chain of strict extensions. These are properties of the internal model, not appeals to the ambient set-theoretic meaning of finite.

\subsection{Consequences of the all-finite hypothesis}
\begin{lemma}[Finite pigeonhole]\label{lem:pigeon}
Under \eqref{eq:allfinite}, an internal injection from an object $A$ to itself is surjective. Consequently an object cannot be internally bijective with a proper internally represented subset of itself.
\end{lemma}
\begin{proof}[Proof idea, including the formal construction]
If an injection $f\colon A\to A$ omits $a\in A$, take the least internally represented $f$-closed subset containing $a$. On it form the strict forward-reachability relation induced by $f$. Injectivity and the omission of $a$ make reachability irreflexive; reachability is transitive and every point has an $f$-successor. The closure and reachability sets are built by stratified internal codes in the formalization, and they contradict $\NoNT(A)$ (more precisely its restriction to the inhabited closure). For the consequence, compose a bijection $A\to B\subsetneq A$ with inclusion $B\hookrightarrow A$.
\end{proof}

\begin{lemma}[Cardinal comparison and finite extrema]\label{lem:compare}
Assuming \eqref{eq:allfinite}, any two positive Frege cardinals are comparable under internal injection. Every inhabited internal set of such cardinals has a least and a greatest member.
\end{lemma}
\begin{proof}[Proof idea]
Consider the inhabited internal collection of nonempty partial bijections between representatives $A,B$, ordered by proper extension. An internally represented, irreflexive, transitive, nonterminating extension relation would contradict $\NoNT$ of this collection; hence a maximal partial bijection exists. If both its domain and its range were proper, one could adjoin one unused matched pair, contradicting maximality. Thus one representative injects into the other. Antisymmetry follows using internal injection transport and the finite pigeonhole lemma. If an inhabited internal collection of cardinals had no greatest (respectively least) member, the internally coded strict greater-than (respectively smaller-than) relation would be nonterminating, contradicting $\NoNT$.
\end{proof}

These lemmas are instantiated by actual stratified formula codes in the formal development; merely writing ``the set of all partial bijections'' would not by itself establish their availability in this theory.

\section{Positive Frege cardinals and type-raising}\label{sec:card}
For inhabited $A$ the \emph{positive Frege cardinal} is
\[
|A|\defeq\{B:B\simeq A\}.
\]
This class has representative $A$ and hence is an object by the internal coding argument. Write $1_F=|\{a\}|$ and $2_F=|\{a,b\}|$ for $a\neq b$. Their definitions are independent of the chosen representatives. We use $p<q$ for internal injection in one direction but not equinumerosity.

Set $\USC(A)\defeq\{\{x\}:x\in A\}$ and define type-raising by
\begin{equation}\label{eq:T}
T(|A|)\defeq|\USC(A)|.
\end{equation}
For $A$ and any fresh $p$, write $|A|+1_F$ for $|A\cup\{p\}|$; in the argument this is the internal successor relation, not a postulated arithmetic operator on external natural numbers. The formalization proves:
\begin{equation}\label{eq:Tlaws}
T(1_F)=1_F,\quad T(2_F)=2_F,\quad
p<q\Rightarrow T(p)<T(q),\quad
T(p+1_F)=T(p)+1_F
\end{equation}
whenever the indicated expressions are represented. No global axiom $T(p)=p$ is used; in NF this would be an illicit use of the Axiom of Counting.

\section{Positive exponentiation and its boundary}\label{sec:exponent}
For inhabited $A$ define the \emph{positive power object}
\[
\Powp(A)\defeq\{X:X\subseteq A\text{ and }\exists x\in X\},
\]
and the partial Frege-cardinal exponentiation $E$ by
\begin{equation}\label{eq:E}
E\bigl(|\USC(A)|\bigr)=|\Powp(A)|.
\end{equation}
The qualification ``partial'' matters: in general not every positive Frege cardinal is represented in the form $|\USC(A)|$. The source formalization supplies an actual internal relation graph for $E$, proves functionality, and verifies the properties below. The use of the nonempty power object ensures the defining comprehension is inhabited without assuming an empty set.

\begin{lemma}[Positive exponentiation laws]\label{lem:exp}
Under \eqref{eq:allfinite}, for represented arguments above $1_F$:
\begin{align}
 E(p)\downarrow\ \&\Longrightarrow\ p<E(p),\label{eq:exp-growth}\\
 1_F<p<q,\quad E(p),E(q)\downarrow\ \&\Longrightarrow\ E(p)<E(q).\label{eq:exp-mono}
\end{align}
There is also the typed compatibility $E(T(p))=T(E(p))$ whenever the expressions are defined in the required internal sense.
\end{lemma}
\begin{proof}[Principal cardinal-comparison argument]
Choose $p=|\USC(A)|$ and $q=|\USC(B)|$, and decode an internal injection of singleton representatives into $f:A\hookrightarrow B$. If $p<q$, its image $I=f[A]$ is proper in $B$. The direct-image operation gives an internal bijection $\Powp(A)\simeq\Powp(I)$. Moreover $\Powp(I)\subsetneq\Powp(B)$: for $b\in B\setminus I$ its singleton lies only in the latter. The finite pigeonhole lemma rules out equal cardinalities, giving \eqref{eq:exp-mono}. For strict growth, the singleton subsets of a non-singleton $A$ constitute a proper internally represented subset of $\Powp(A)$ (the set $A$ itself lies outside that singleton family). Identifying $\USC(A)$ with this family gives $p<E(p)$ by finite pigeonhole. Type-compatibility is a separate internally represented power-object recoding lemma in the formalization; its stratification is not inferred from cardinal comparison alone.
\end{proof}

Set
\[
\Omega\defeq|V|,\qquad \kappa\defeq T(\Omega)=|\USC(V)|,
\qquad \lambda\defeq T(\kappa).
\]
The exact boundary lemma is:
\begin{lemma}[Exponentiation boundary]\label{lem:boundary}
Under \eqref{eq:allfinite}, the represented positive exponentiation relation satisfies
\begin{equation}\label{eq:boundary}
E(\kappa)=\Omega,\qquad E(\lambda)=\kappa,
\qquad E(p)\downarrow\iff p\leq\kappa,
\qquad T(p)\leq\kappa.
\end{equation}
In particular $E(\Omega)$ is undefined.
\end{lemma}
\begin{proof}[Explanation of the boundary]
Since every object is a nonempty subset of $V$, $\Powp(V)=V$; this yields $E(\kappa)=\Omega$. A nonempty $X\subseteq\USC(V)$ is the singleton image of its decoded set $A=\bigcup X$. The homogeneous internal recoding $X\mapsto\{A\}$ identifies $\Powp(\USC(V))$ with $\USC(V)$, yielding $E(\lambda)=\kappa$. Every singleton image injects into $\USC(V)$, explaining $T(p)\leq\kappa$. Conversely, given $p\leq\kappa$, an internally represented injection into $\USC(V)$ and its nonempty image give, by decoding, a representative $A$ with $p=|\USC(A)|$. This is the domain characterization. Finally $\USC(V)$ is a proper subset of $V$ (there is a genuine two-element object), so $\kappa<\Omega$ by finite pigeonhole. Therefore $E(\Omega)$ is outside the domain. In particular, the equality $E(\lambda)=\kappa$ is a \emph{typed recoding} assertion and must not be replaced by the naive untyped identity $\Powp(\USC(V))=V$.
\end{proof}

\section{The least exponentiation orbit}\label{sec:orbit}
For each positive cardinal $p$ let $\Orb(p)$ be the \emph{least internally represented} set of positive cardinals containing $p$ and closed under the internal graph of $E$. This is constructed by a stratified intersection/closure formula with a specific witness in the formal development, not by an unrestricted second-order intersection. Informally its elements are $p,E(p),E^2(p),\ldots$ until exponentiation becomes undefined.

\begin{lemma}[Orbit splitting]\label{lem:split}
If $p>1_F$ and $E(p)$ is undefined, then $\Orb(p)=\{p\}$. If $q=E(p)$ is defined, then
\begin{equation}\label{eq:split}
\Orb(p)=\{p\}\mathbin{\dot\cup}\Orb(q),\qquad
|\Orb(p)|=|\Orb(q)|+1_F.
\end{equation}
The corresponding decomposition for $\Orb(T(p))$ is valid when its edge exists.
\end{lemma}
\begin{proof}
Strict growth prevents the orbit from returning to $p$. Closure shows that $\{p\}\cup\Orb(q)$ is an $E$-closed set containing $p$. Leastness gives one inclusion; closure of $\Orb(p)$ gives the reverse inclusion. The case without a successor has only its seed. Both arguments refer to the internal least-closure representation, and their Frege-cardinal forms use internally constructed one-point extension graphs.
\end{proof}

\begin{lemma}[Exact terminal orbit]\label{lem:terminal}
If $p>1_F$ and $E(p)$ is undefined, then
\begin{equation}\label{eq:terminal}
|\Orb(p)|=1_F,\qquad |\Orb(T(p))|=2_F.
\end{equation}
\end{lemma}
\begin{proof}
By Lemma~\ref{lem:boundary}, $p>\kappa$ (using comparison). Strict monotonicity of $T$ gives $T(p)>T(\kappa)=\lambda$, while $T(p)\leq\kappa$ ensures that $E(T(p))$ exists. From $E(\lambda)=\kappa$ and strict monotonicity of $E$ we obtain
\[
E(T(p))>E(\lambda)=\kappa.
\]
Hence $E(E(T(p)))$ is undefined. The type-raised orbit has exactly the two distinct elements $T(p)$ and $E(T(p))$; the original orbit is the singleton $\{p\}$. The internal orbit-splitting laws give \eqref{eq:terminal}.
\end{proof}

\section{The single-orbit cardinal equation}\label{sec:equation}
Construct a genuine two-element set $A=\{a,b\}$ with $a\neq b$; take the Frege representative $\USC(A)$ and its cardinal $2_F$. The positive power object $\Powp(A)$ and the internal exponentiation graph provide an outgoing edge from this seed. Let
\[
C\defeq\Orb(2_F).
\]
By \eqref{eq:allfinite}, $C$ is finite.

\begin{lemma}[The type-raised orbit identity]\label{lem:orbitidentity}
For every $p\in C$,
\begin{equation}\label{eq:recursion}
|\Orb(T(p))|=T\bigl(|\Orb(p)|\bigr)+1_F.
\end{equation}
\end{lemma}
\begin{proof}
Define $P(p)$ to be \eqref{eq:recursion} and take the internally represented set of counterexamples in $C$. If it were inhabited, its finite greatest element $p$ would exist by Lemma~\ref{lem:compare}. If $E(p)$ is undefined, Lemma~\ref{lem:terminal} gives $|\Orb(T(p))|=2_F=T(1_F)+1_F$, proving $P(p)$. Otherwise let $q=E(p)>p$; then $q\in C$, so maximality gives $P(q)$. Orbit splitting and $T$-compatibility yield
\begin{align*}
 |\Orb(T(p))|
 &=|\Orb(T(q))|+1_F\\
 &=T\bigl(|\Orb(q)|\bigr)+1_F+1_F\\
 &=T\bigl(|\Orb(q)|+1_F\bigr)+1_F\\
 &=T\bigl(|\Orb(p)|\bigr)+1_F,
\end{align*}
contradicting the choice of $p$. The complete Lean proof supplies a stratified formula for the counterexample set and the necessary internal Frege representative transports.
\end{proof}

\begin{corollary}[Decisive equation]\label{cor:equation}
Writing $n\defeq|C|$, one has
\begin{equation}\label{eq:decisive}
\boxed{\quad n=T(n)+1_F.\quad}
\end{equation}
\end{corollary}
\begin{proof}
Apply Lemma~\ref{lem:orbitidentity} at $p=2_F$. Since $T(2_F)=2_F$, the two orbit arguments agree; hence $|\Orb(T(2_F))|=|C|=n$.
\end{proof}

\section{Detailed parity proof}\label{sec:parity}
This section supplies the crucial step that an abbreviated presentation might conceal. It is essential to use \emph{internally represented} bijection graphs and an \emph{internally coded} induction predicate, rather than relying on ordinary arithmetic induction for external natural numbers.

\begin{definition}[Internal evenness]\label{def:even}
An inhabited internal set $A$ is \emph{even}, written $\Even(A)$, if there are internal sets $X,Y$ and an internal bijection graph $f:X\simeq Y$ such that
\begin{equation}\label{eq:even}
A=X\mathbin{\dot\cup}Y.
\end{equation}
In the no-empty model both halves are inhabited. We use $\neg\Even(A)$ as the contrasting parity; no independent ``odd partition'' predicate or parity-exhaustion theorem is necessary.
\end{definition}

Write $A\rightsquigarrow B$ when, for a fresh point $p\notin A$, $B=A\cup\{p\}$ internally; the fresh point is part of the assertion. Two consecutive extensions $A\rightsquigarrow B\rightsquigarrow C$ add distinct fresh points $p,q$.

\begin{lemma}[Two-point invariance]\label{lem:twopoint}
If $A\rightsquigarrow B\rightsquigarrow C$, then
\begin{equation}\label{eq:two}
\Even(A)\iff\Even(C).
\end{equation}
\end{lemma}
\begin{proof}
\emph{Forward implication.} Suppose $A=X\dot\cup Y$ and $f:X\simeq Y$. Insert the first new point $p$ into $X$ and the second $q$ into $Y$; extend the graph of $f$ by the edge $p\mapsto q$. This graph represents an internal bijection witnessing that $C$ is even. The extension code has the new edge as its explicit inhabitance witness for comprehension.

\emph{Reverse implication.} Suppose $C=X\dot\cup Y$ with bijection $f:X\to Y$. Remove the newly adjoined distinct points $p,q$, leaving the inhabited set $A=C\setminus\{p,q\}$. There are three substantive arrangements, up to interchanging $X,Y$.

\smallskip\noindent\textbf{Case 1: both removed points lie in the same half.}
Suppose $p,q\in X$. Let $v=f(p)$ and $w=f(q)$; injectivity gives $v\neq w$. Define
\[
X'=(X\setminus\{p,q\})\cup\{v\},\qquad Y'=Y\setminus\{v\}.
\]
The surviving carrier $A$ is exactly $X'\dot\cup Y'$. Define $f':X'\to Y'$ by retaining the old edges at $x\in X\setminus\{p,q\}$ and replacing them by the single edge $v\mapsto w$. Old surviving images avoid $v,w$, hence $f'$ is a bijection. The symmetric subcase $p,q\in Y$ is obtained by using the internally constructed inverse graph $f^{-1}$.

\smallskip\noindent\textbf{Case 2: the removed points lie in opposite halves and are matched.}
Suppose $p\in X$, $q\in Y$, and $f(p)=q$. Take
\[
X'=X\setminus\{p\},\qquad Y'=Y\setminus\{q\},
\]
and restrict $f$ to the surviving sources. It is a bijection onto $Y'$. As $A$ is inhabited and consists of the disjoint union of these equal-sized surviving halves, each half is inhabited, as required by our restricted comprehension principle.

\smallskip\noindent\textbf{Case 3: the removed points lie in opposite halves but are not matched.}
Suppose $p\in X$, $q\in Y$, and $f(p)=v\neq q$. Let $u=f^{-1}(q)\neq p$. Again set $X'=X\setminus\{p\}$ and $Y'=Y\setminus\{q\}$, but repair the matching by
\begin{equation}\label{eq:reroute}
 f'\defeq\bigl(f\upharpoonright(X\setminus\{p,u\})\bigr)\cup\{(u,v)\}.
\end{equation}
Thus the old edges $p\mapsto v$ and $u\mapsto q$ are removed and the edge $u\mapsto v$ is inserted. The remaining arrows are unchanged. This is an internally represented bijection $X'\to Y'$.

In all cases the required nonempty pieces and replacement graph are produced by witnessed, stratified comprehensions. In particular the two-point construction never postulates an empty relation set. These are the explicit graph-surgery cases formalized in Runs~64--65, proving the converse and hence \eqref{eq:two}.
\end{proof}

\begin{lemma}[One-point reversal at a singleton]\label{lem:base}
If $A=\{p\}$ and $A\rightsquigarrow B$, then
\begin{equation}\label{eq:singleflip}
\Even(A)\iff\neg\Even(B).
\end{equation}
\end{lemma}
\begin{proof}
$A$ cannot be split into two disjoint inhabited equinumerous halves. $B$ is a doubleton, partitioned into its two singletons with an internal one-edge bijection. Thus $\neg\Even(A)$ and $\Even(B)$, exactly \eqref{eq:singleflip}.
\end{proof}

\begin{lemma}[Induction from the all-finite hypothesis]\label{lem:induction}
Assume \eqref{eq:allfinite}. Let $P(A)$ be represented by an actual stratified first-order formula of the membership language. If $P$ holds on every singleton and is inherited by every internally represented fresh one-point extension, then $P(A)$ holds for every object $A$.
\end{lemma}
\begin{proof}
If there is a counterexample, inhabited stratified comprehension produces the \emph{internal} counterexample set $D=\{A:\neg P(A)\}$. No element of $D$ is a singleton. Every $B\in D$ therefore contains distinct $p,q$; remove $p$ using $q$ as the explicit witness that $B\setminus\{p\}$ exists and is inhabited. If that predecessor satisfied $P$, the extension step would imply $P(B)$, a contradiction. Consequently every $B\in D$ has a proper predecessor in $D$. The relation $B\,R\,A$ defined by ``$A$ is a proper subset of $B$'' is irreflexive, transitive and nonterminating on $D$. The formal proof constructs its internal Kuratowski graph by a stratified formula and restricts its support to $D$. This contradicts $\NoNT(D)$ in \eqref{eq:allfinite}. Hence $D$ cannot exist.
\end{proof}

\begin{lemma}[One-point parity reversal]\label{lem:flip}
Under \eqref{eq:allfinite}, every fresh one-point extension satisfies
\begin{equation}\label{eq:flip}
A\rightsquigarrow B\quad\Longrightarrow\quad
\Even(A)\iff\neg\Even(B).
\end{equation}
\end{lemma}
\begin{proof}
For each internal set $A$ consider the statement
\[
P(A)\defeq\forall B,p\;\bigl((B=A\cup\{p\}\ \wedge\ p\notin A)
\Longrightarrow(\Even(A)\leftrightarrow\neg\Even(B))\bigr).
\]
The singleton base is Lemma~\ref{lem:base}. For the successor step, suppose $P(A)$ and $A\rightsquigarrow B\rightsquigarrow C$. From $P(A)$ we have
\[
\Even(A)\iff\neg\Even(B).
\]
From Lemma~\ref{lem:twopoint} we have
\[
\Even(A)\iff\Even(C).
\]
Eliminating $\Even(A)$ yields $\Even(B)\iff\neg\Even(C)$, which is the required $P(B)$. Finally, unlike unrestricted induction over arbitrary predicates, Lemma~\ref{lem:induction} applies here \emph{because the formalization provides an explicit stratified formula for $P$ and proves its intended semantics}. The relevant formula and correctness theorem are in Run~67. Therefore $P(A)$ holds for every $A$, and \eqref{eq:flip} follows.
\end{proof}

\begin{lemma}[Parity transport to Frege cardinals]\label{lem:fregeparity}
Define $\Even(|A|)$ by existence of an even representative. This is independent of the representative. For positive Frege cardinals $r,s$:
\begin{align}
\Even(T(r))&\iff\Even(r),\label{eq:parityT}\\
\Succ(r,s)&\Longrightarrow
\bigl(\Even(r)\iff\neg\Even(s)\bigr).\label{eq:paritysucc}
\end{align}
\end{lemma}
\begin{proof}
An internal bijection $A\simeq B$ transports a bipartition and its matching to $B$, so $\Even$ depends only on $|A|$. Taking singleton images of both halves and sending $\{x\}$ to $\{f(x)\}$ transports a matching to $\USC(A)$, proving the forward direction of \eqref{eq:parityT}; the inverse argument decodes the singleton representatives. An internally coded Frege-successor witness consists of representatives $A,B$ and a fresh one-point extension $A\rightsquigarrow B$. Apply Lemma~\ref{lem:flip} and representative independence to obtain \eqref{eq:paritysucc}.
\end{proof}

\section{Conclusion of the contradiction}\label{sec:end}
\begin{proof}[Proof of Theorem~\ref{thm:main}]
Suppose for a contradiction that every object is finite. Sections~\ref{sec:card}--\ref{sec:equation} construct the actual finite Frege cardinal $n=|\Orb(2_F)|$ and prove $n=T(n)+1_F$. Lemma~\ref{lem:fregeparity} says simultaneously
\[
\Even(n)\iff\Even(T(n))
\quad\text{and}\quad
\Even(n)\iff\neg\Even(T(n)),
\]
the second equivalence following from the successor equation. Propositional logic now yields a contradiction without choosing an ``odd'' witness or postulating parity exhaustiveness. Thus $\neg\forall A\,\Fin(A)$, and classical logic gives $\exists A\,\neg\Fin(A)$.
\end{proof}

\section{Relation to the Lean formalization}\label{sec:lean}
The companion Lean~4 project (pinned to Lean~4.24.0 and Mathlib~v4.24.0) encodes the above theory as an explicit structure rather than treating the model's existence as a theorem. In particular, the top-level statement is:
\begin{quote}\ttfamily\small
NF0Specker.Run68.infinity\\
\{V : Type*\} (M : NF0Specker.Run17.PositiveModel V) :\\
\hspace*{1em}\(\exists\) S : V, \(\neg\) NF0Specker.Run17.Finite M.mem S
\end{quote}
The supplied archive includes successful build and axiom-audit logs. Its verification report says the main and separate Tarski appendix build under the pinned toolchain. The accompanying report is evidence of a previous successful compilation; the present mathematical paper was not independently recompiled as part of typesetting.

\begin{center}\footnotesize
\begin{tabular}{@{}p{0.41\linewidth}p{0.51\linewidth}@{}}\toprule
Mathematical obligation & Primary Lean declaration(s)\\\midrule
Exact model and finiteness & \texttt{Run17.PositiveModel}, \texttt{Run17.Finite}\\
Pigeonhole and comparison & \texttt{endomap\_surjective\_of\_all\_finite};\newline \texttt{run35\_all\_positive\_frege\_comparable}\\
Finite extrema & \texttt{Run17.maximal\_of\_supportedNoNT} and applications\\
Exponentiation graph/boundaries & \texttt{positiveExponentGraph\_with\_boundaries\_exists}\\
Two-point graph surgery & \texttt{run64\_edgeRerouteGraph\_exists};\newline \texttt{run65\_reverse\_two\_point\_even};\newline \texttt{run65\_two\_extensions\_even\_iff}\\
Stratified parity induction & \texttt{run66\_stratified\_oldSet\_induction};\newline \texttt{run67EvenCode\_correct};\newline \texttt{run67FlipCode\_correct}\\
Successor parity flip & \texttt{run67\_one\_point\_flip};\newline \texttt{run63\_frege\_successor\_flip}\\
Exact two-seed assembly & \texttt{Run68.all\_finite\_contradiction}; \texttt{Run68.infinity}\\\bottomrule
\end{tabular}
\end{center}

The recorded axiom audit lists exactly the following Lean metatheoretic dependencies:\begin{center}\texttt{propext}\qquad\texttt{Classical.choice}\qquad\texttt{Quot.sound}\end{center}No other axiom or admitted proof hole is reported in the endpoint's transitive dependencies. These dependencies do \emph{not} add an object-level Choice axiom to the modeled set theory. The source files retain some historical comments calling earlier stages ``uncompiled''; verification status should be judged from the final build report and reproducible compilation, not those superseded comments.

\subsection{Scope and limitations}
The paper is a mathematical exposition of the proof architecture and principal internal constructions. It does not print every first-order formula AST, every stratification certificate, or all the graph-transport derivations. Those details belong to the accompanying Lean source. Conversely, a Lean kernel check establishes the theorem with the \emph{exact} model fields and definitions encoded in that source: it must not be paraphrased as a theorem of a different axiom system without an interpretation proof. The Tarski-finiteness equivalence present in the separate appendix is useful additional context, not an ingredient of the main endpoint.

\section*{Acknowledgement and source transparency}
This exposition is based on the accompanying mathematical notes and Lean formal development. Assistance with mathematical exposition and manuscript preparation was provided by an AI writing tool. The named author retains responsibility for the mathematical content and scholarly attribution.

\end{document}